\documentclass[reqno,11pt]{article} 
\usepackage{amsmath, amssymb}
\usepackage{amsthm} 
\newcommand{\Mod}[1]{\ (\textup{mod}\ #1)}
\theoremstyle{plain} 
\newtheorem{theorem}{\indent\sc Theorem}[section]
\newtheorem{lemma}[theorem]{\indent\sc Lemma}

\newtheorem{proposition}[theorem]{\indent\sc Proposition}

\theoremstyle{definition} 

\newtheorem{remark}[theorem]{\indent\sc Remark}

\newtheorem{example}[theorem]{\indent\sc Example}

\makeatletter
\def\address#1#2{\begingroup
\noindent\parbox[t]{7.8cm}{%
\small{\scshape\ignorespaces#1}\par\vskip1ex
\noindent\small{\itshape E-mail address}%
\/: #2\par\vskip4ex}\hfill%
\endgroup}%
\makeatother
\title{Primitive and totally primitive Fricke families with applications} 
\author{
\textsc{Ho Yun Jung, Ja Kyung Koo and Dong Hwa Shin} 
}
\date{} 
\begin{document}

\allowdisplaybreaks

\maketitle

\footnote{
\thanks{
The first named author was supported by the National Research Foundation
of Korea (NRF) grant funded by the Korea government (MSIP) (2016R1A5A1008055).
The third named (corresponding) author was supported by Hankuk University of Foreign Studies Research Fund of 2016.} }

\begin{abstract}
We introduce the primitivity of Fricke families, and give some examples.
As its application, we first construct
generators of the function field of
the modular curve of level $N$
in terms of Fricke functions and Siegel functions, respectively.
Furthermore, we use the special values of a certain function in a totally primitive Fricke family of level $N$ in order to generate ray class fields of imaginary quadratic fields.
\end{abstract}

\section {Introduction}

For a positive integer $N$, let
$\Gamma(N)=\left\{\gamma\in\mathrm{SL}_2(\mathbb{Z})~|~\gamma\equiv
I_2\Mod{N}\right\}$
be the principal congruence subgroup of $\mathrm{SL}_2(\mathbb{Z})$ of level $N$. This group
acts on the complex upper half-plane $\mathbb{H}=\{\tau\in\mathbb{C}~|~
\mathrm{Im}(\tau)>0\}$ and
$\mathbb{H}^*=\mathbb{H}\cup\mathbb{Q}\cup
\{\mathrm{i}\infty\}$
as fractional linear transformations. One can then give
the orbit space $X(N)=\Gamma(N)\backslash\mathbb{H}^*$ the structure of
a compact Riemann surface, called
the \textit{modular curve} of level $N$ (\cite[$\S$1.5]{Shimura}).
Let $\mathbb{C}(X(N))$ be the field of meromorphic functions
on $X(N)$ which is a Galois extension of $\mathbb{C}(X(1))=\mathbb{C}(j(\tau))$ with
\begin{equation*}
\mathrm{Gal}\left(\mathbb{C}(X(N))/\mathbb{C}(X(1))\right)\simeq
\mathrm{SL}_2(\mathbb{Z})/\pm\Gamma(N)\simeq\mathrm{SL}_2(\mathbb{Z}/N\mathbb{Z})/\{\pm I_2\},
\end{equation*}
where $j(\tau)$ is the elliptic modular function (\cite[Theorem 2 in Chapter 6]{Lang}).
Furthermore, we denote by $\mathcal{F}_N$ the subfield of $\mathbb{C}(X(N))$ consisting of
functions whose Fourier coefficients lie in the $N$th cyclotomic field
$\mathbb{Q}(\zeta_N)$, where $\zeta_N=e^{2\pi\mathrm{i}/N}$. Then, $\mathcal{F}_N$ is also a Galois extension of $\mathcal{F}_1=\mathbb{Q}(j(\tau))$ whose Galois group is isomorphic to
$\mathrm{GL}_2(\mathbb{Z}/N\mathbb{Z})
/\{\pm I_2\}$
(see $\S$\ref{sect2}).
\par
For $N\geq2$, let
\begin{equation*}
\mathcal{V}_N=\{\mathbf{v}\in\mathbb{Q}^2~|~
\textrm{$N$ is the least positive integer so that}~N\mathbf{v}\in\mathbb{Z}^2\}.
\end{equation*}
We call a family $\{h_\mathbf{v}(\tau)\}_{\mathbf{v}\in\mathcal{V}_N}$
of functions in $\mathcal{F}_N$ a
\textit{Fricke family} of level $N$ if it satisfies the following three conditions:
\begin{itemize}
\item[(F1)] Every $h_\mathbf{v}(\tau)$ is holomorphic on $\mathbb{H}$.
\item[(F2)] $h_\mathbf{u}(\tau)=h_\mathbf{v}(\tau)$ if $\mathbf{u}\equiv\pm\mathbf{v}
\Mod{\mathbb{Z}^2}$.
\item[(F3)] $h_\mathbf{v}(\tau)^\alpha=h_{\alpha^T\mathbf{v}}(\tau)$
for $\alpha\in\mathrm{GL}_2(\mathbb{Z}/N\mathbb{Z})/\{\pm I_2\}$,
where $\alpha^T$ stands for the transpose of $\alpha$.
\end{itemize}
As for a Fricke family, Kubert and Lang first gave its definition without the condition (F1) (\cite[pp. 32--33]{K-L}).
Recently, Eum and Shin (\cite{E-K-S})
classified all Fricke families
of level $N$ when $N\equiv0\Mod{4}$. See also \cite{K-Y2}.
\par
We say that a Fricke family $\{h_\mathbf{v}(\tau)\}_{\mathbf{v}\in\mathcal{V}_N}$ of level $N$ is \textit{primitive} if the condition (F2) is strengthened in such a way that
\begin{equation*}
h_\mathbf{u}(\tau)=h_\mathbf{v}(\tau)~
\Longleftrightarrow~
\mathbf{u}\equiv\pm\mathbf{v}\Mod{\mathbb{Z}^2}.
\end{equation*}
Moreover, we say that
$\{h_\mathbf{v}(\tau)\}_{\mathbf{v}\in\mathcal{V}_N}$
is \textit{totally primitive}
if $\{h_\mathbf{v}(\tau)^n\}_{\mathbf{v}\in\mathcal{V}_N}$
is primitive for every positive integer $n$.
In this paper, we shall present several examples of Fricke families
which are primitive or totally primitive (Examples \ref{eg1}, \ref{eg2} and \ref{eg3}).
\par
As is well known, we have
\begin{equation*}
\mathbb{C}(X(N))=\mathbb{C}\left(j(\tau),\,f_{\left[\begin{smallmatrix}1/N\\0\end{smallmatrix}\right]}(\tau),\,
f_{\left[\begin{smallmatrix}0\\1/N\end{smallmatrix}\right]}(\tau)\right),
\end{equation*}
where $f_\mathbf{v}(\tau)$ ($\mathbf{v}\in\mathcal{V}_N)$ are
the classical Fricke functions (see $\S$2 and \cite[Proposition 7.5.1]{D-S}).
Since the modular curve $X(N)$ is an algebraic curve, its function field
$\mathbb{C}(X(N))$ can be generated by two functions (\cite[Theorem 1.9 and Proposition 1.17 in Chapter VI]{Miranda}). As an application of primitive Fricke families,
we shall first construct a primitive generator of $\mathbb{C}(X(N))$ over the field $\mathbb{C}(X(1))
=\mathbb{C}(j(\tau))$ in terms of Fricke functions $f_{\left[\begin{smallmatrix}1/N\\0\end{smallmatrix}\right]}(\tau)$
and $f_{\left[\begin{smallmatrix}0\\1/N\end{smallmatrix}\right]}(\tau)$
(Theorem \ref{usingFricke}) which belong to a primitive Fricke family.
We shall further present a primitive generator of $\mathbb{C}(X(N))$ over
$\mathbb{C}(X(1))$ by making use of only Siegel functions
as members of a totally primitive Fricke family (Theorem \ref{usingSiegel}
and Remark \ref{singularmodel}).
\par
Let $K$ be an imaginary quadratic field of discriminant $d_K$, and let
$\mathcal{O}_K$ be its ring of integers.
If we set
\begin{equation*}
\tau_K=(d_K+\sqrt{d_K})/2,
\end{equation*}
then we see that $\tau_K\in\mathbb{H}$ and $\mathcal{O}_K=\mathbb{Z}\tau_K+\mathbb{Z}$ (\cite[$\S$5.B]{Cox}).
By $H_K$ we mean the Hilbert class field of $K$, and
by $K_{(N)}$ the ray class field modulo $N\mathcal{O}_K$.
Let $\{h_\mathbf{v}(\tau)\}_{\mathbf{v}\in\mathcal{V}_N}$ be a totally primitive Fricke family of level $N$.
For all but finitely many $K$, we shall show that
if the special value $h_{\left[\begin{smallmatrix}0\\1/N\end{smallmatrix}\right]}(\tau_K)$
is nonzero, then
$h_{\left[\begin{smallmatrix}0\\1/N\end{smallmatrix}\right]}(\tau_K)^n$
generates $K_{(N)}$ over $H_K$ for any nonzero integer $n$ (Theorem \ref{invariant} and Remark
\ref{invariantremark}).
\par
Based on this work, Koo et al. established the concept of
a (totally) primitive Siegel family consisting of meromorphic Siegel modular functions
of higher genus $g$ ($\geq2$) (\cite[Definition 3.1]{K-S-Y}).
They further constructed explicit generators of the field of
Siegel modular functions of level $N$ ($\neq 2,\,2^g-1,\,2(2^g-1)$) over the field of Siegel modular functions of level $1$ (\cite[Proposition 3.3 and Theorem 6.2]{K-S-Y}).
To this end, they reduced each theta constant
of genus $g$ to a product of Siegel functions of one-variable, and
then made use of the idea of Example \ref{eg1}.
We also notice that there is a recent attempt (\cite{K-R-S-Y})
to get a higher genus version of Theorem \ref{invariant} for CM-fields.

\section {Meromorphic modular functions}\label{sect2}

Let $N$ be a  positive integer.
The group
$\mathrm{GL}_2(\mathbb{Z}/N\mathbb{Z})/\{\pm I_2\}$ ($\simeq\mathrm{Gal}(\mathcal{F}_N/\mathcal{F}_1)$) acts on the field
$\mathcal{F}_N$ as follows (\cite[Theorem 3 in Chapter 6]{Lang}): One
can decompose $\mathrm{GL}_2(\mathbb{Z}/N\mathbb{Z})/\{\pm I_2\}$ uniquely as
\begin{equation*}
\mathrm{GL}_2(\mathbb{Z}/N\mathbb{Z})/\{\pm I_2\}=G_N\cdot
\mathrm{SL}_2(\mathbb{Z}/N\mathbb{Z})/\{\pm I_2\}
~\textrm{with}~
G_N=\left\{\begin{bmatrix}1&0\\0&d\end{bmatrix}~|~
d\in(\mathbb{Z}/N\mathbb{Z})^\times\right\}.
\end{equation*}
Let $h(\tau)$ be an element of $\mathcal{F}_N$ whose Fourier expansion with respect to
$q^{1/N}=e^{2\pi\mathrm{i}\tau/N}$ is given by
\begin{equation*}
h(\tau)=\sum_{n\gg-\infty}c_nq^{n/N}\quad(c_n\in\mathbb{Q}(\zeta_N)).
\end{equation*}
\begin{itemize}
\item[(A1)] $\begin{bmatrix}1&0\\0&d\end{bmatrix}\in G_N$
acts on $h(\tau)$ as
\begin{equation*}
h(\tau)^{\left[\begin{smallmatrix}1&0\\0&d\end{smallmatrix}\right]}=
\sum_{n\gg-\infty}c_n^{\sigma_d}q^{n/N},
\end{equation*}
where $\sigma_d$ is the automorphism of the cyclotomic field $\mathbb{Q}(\zeta_N)$
determined by $\zeta_N^{\sigma_d}=\zeta_N^d$.
\item[(A2)] $\alpha\in\mathrm{SL}_2(\mathbb{Z}/N\mathbb{Z})/\{\pm I_2\}$
acts on $h(\tau)$ by
\begin{equation*}
h(\tau)^\alpha=(h\circ\widetilde{\alpha})(\tau),
\end{equation*}
where $\widetilde{\alpha}$ is any inverse image of $\alpha$ under the reduction
$\mathrm{SL}_2(\mathbb{Z})\rightarrow\mathrm{SL}_2(\mathbb{Z}/N\mathbb{Z})/\{\pm I_2\}$.
\end{itemize}
\par
For a lattice $\Lambda$ in $\mathbb{C}$, let
\begin{equation*}
g_2(\Lambda)=60\sum_{\lambda\in\Lambda\setminus\{0\}}\frac{1}{\lambda^4},\quad
g_3(\Lambda)=140\sum_{\lambda\in\Lambda\setminus\{0\}}\frac{1}{\lambda^6}
\quad\textrm{and}\quad
\Delta(\Lambda)=g_2(\Lambda)^3-27g_3(\Lambda)^2.
\end{equation*}
The \textit{Weierstrass $\wp$-function} relative to $\Lambda$ is defined by
\begin{equation*}
\wp(z;\,\Lambda)=
\frac{1}{z^2}+\sum_{\lambda\in\Lambda\setminus\{0\}}\left\{\frac{1}{(z-\lambda)^2}-\frac{1}{\lambda^2}
\right\}
\quad(z\in\mathbb{C})
\end{equation*}
with a double pole at each lattice point, and no other poles (\cite[p. 8]{Lang}).
By the \textit{Weierstrass $\sigma$-function} relative to $\Lambda$
we mean the infinite product
\begin{equation*}
\sigma(z;\,\Lambda)=z\prod_{\lambda\in \Lambda\setminus\{0\}}\left(1-\frac{z}{\lambda}\right)
e^{z/\lambda+(1/2)(z/\lambda)^2}\quad(z\in\mathbb{C}).
\end{equation*}
Taking logarithmic derivative, we derive the \textit{Weierstrass
$\zeta$-function}
\begin{equation*}
\zeta(z;\,\Lambda)=\frac{\sigma'(z;\,\Lambda)}{\sigma(z;\,\Lambda)}
=\frac{1}{z}+\sum_{\lambda\in
\Lambda\setminus\{0\}}\left(\frac{1}{z-\lambda}+\frac{1}{\lambda}+
\frac{z}{\lambda^2}\right)\quad(z\in\mathbb{C}).
\end{equation*}
Since $\zeta'(z;\,\Lambda)=-\wp(z;\,\Lambda)$ which is periodic with respect to $\Lambda$,
for each $\lambda\in\Lambda$ we obtain a constant $\eta(\lambda;\,\Lambda)$ satisfying
\begin{equation*}
\zeta(z+\lambda;\,\Lambda)-\zeta(z;\,\Lambda)=\eta(\lambda;\,\Lambda)\quad(z\in\mathbb{C}).
\end{equation*}
\par
Now, let $\mathbf{v}=\begin{bmatrix}v_1\\v_2
\end{bmatrix}\in\mathbb{Q}^2\setminus\mathbb{Z}^2$. We define
the \textit{Fricke function}
$f_\mathbf{v}(\tau)$ by
\begin{equation}\label{Fricke}
f_\mathbf{v}(\tau)
=-2^73^5\frac{g_2(\tau)g_3(\tau)}{\Delta(\tau)}\wp_\mathbf{v}(\tau)
\quad(\tau\in\mathbb{H}),
\end{equation}
where
$g_2(\tau)=g_2([\tau,\,1])$,
$g_3(\tau)=g_3([\tau,\,1])$,
$\Delta(\tau)=\Delta([\tau,\,1])$ and
$\wp_\mathbf{v}(\tau)=\wp(v_1\tau+v_2;\,[\tau,1])$.
Note that $g_2(\tau)$, $g_3(\tau)$ and $\Delta(\tau)$ are holomorphic on
$\mathbb{H}$, and $\Delta(\tau)$ has no zeros on $\mathbb{H}$
(\cite[Theorem 3 in Chapter 3]{Lang}).
We also define
the \textit{Siegel function} $g_\mathbf{v}(\tau)$ by
\begin{equation}\label{Siegel}
g_\mathbf{v}(\tau)=
e^{-(v_1\eta(\tau;\,[\tau,\,1])+
v_2\eta(1;\,[\tau,\,1]))
(v_1\tau+v_2)/2}\sigma(v_1\tau+v_2;\,[\tau,\,1])\eta(\tau)^2
\quad(\tau\in\mathbb{H}),
\end{equation}
where
\begin{equation*}
\eta(\tau)=\sqrt{2\pi}\zeta_8q^{1/24}\prod_{n=1}^\infty
(1-q^n)\quad(\tau\in\mathbb{H})
\end{equation*}
is the \textit{Dedekind $\eta$-function}.
As is well known, if $N\geq2$, then $\{f_\mathbf{v}(\tau)\}_{\mathbf{v}\in\mathcal{V}_N}$ and
$\{g_\mathbf{v}(\tau)^{12N}\}_{\mathbf{v}\in\mathcal{V}_N}$ are Fricke families of level $N$ (\cite[$\S$6.2 and 6.3]{Lang} and \cite[Proposition 1.3 in Chapter 2]{K-L}).
Moreover, $\{f_\mathbf{v}(\tau)\}_{\mathbf{v}\in\mathcal{V}_N}$ is primitive
(\cite[Lemma 10.4]{Cox} and the definition (\ref{Fricke})).
\par
For $x\in\mathbb{R}$, let $\langle x\rangle$ be the fractional part of $x$ in the interval $[0,1)$, and set
\begin{equation*}
\langle\pm x\rangle=\min(\langle x\rangle,\langle-x\rangle).
\end{equation*}
Furthermore, let $\mathbf{B}_2(x)=x^2-x+1/6$
be the second Bernoulli polynomial.

\begin{lemma}\label{porder}
Let $N\geq2$.
\begin{itemize}
\item[\textup{(i)}] If $\mathbf{v}=\begin{bmatrix}v_1\\v_2\end{bmatrix}
\in(1/N)\mathbb{Z}^2\setminus\mathbb{Z}^2$, then we have
$\mathrm{ord}_q~g_\mathbf{v}(\tau)=(1/2)\mathbf{B}_2(\langle v_1\rangle)$.
\item[\textup{(ii)}] Let $\mathbf{u},\,\mathbf{v},\,\mathbf{u}',\,\mathbf{v}'\in(1/N)\mathbb{Z}^2\setminus\mathbb{Z}^2$
such that $\mathbf{u}\not\equiv\pm\mathbf{v}\Mod{\mathbb{Z}^2}$ and
$\mathbf{u}'\not\equiv\pm\mathbf{v}'\Mod{\mathbb{Z}^2}$.
Then, the function
\begin{equation*}
\frac{f_\mathbf{u}(\tau)-f_\mathbf{v}(\tau)}
{f_{\mathbf{u}'}(\tau)-f_{\mathbf{v}'}(\tau)}=
\frac{\wp_\mathbf{u}(\tau)-\wp_\mathbf{v}(\tau)}
{\wp_{\mathbf{u}'}(\tau)-\wp_{\mathbf{v}'}(\tau)}
\end{equation*}
in $\mathcal{F}_N$
has neither zeros nor poles on $\mathbb{H}$.
\item[\textup{(iii)}] If $\mathbf{u}=\begin{bmatrix}u_1\\u_2\end{bmatrix},\,
\mathbf{v}=\begin{bmatrix}v_1\\v_2\end{bmatrix}\in\mathcal{V}_N$ such that $\mathbf{u}+\mathbf{v},\,\mathbf{u}-\mathbf{v}\in\mathcal{V}_N$, then
\begin{equation*}
\mathrm{ord}_q~(\wp_\mathbf{u}(\tau)-\wp_\mathbf{v}(\tau))=
\min(\langle\pm u_1\rangle,\,\langle\pm v_1\rangle).
\end{equation*}
\end{itemize}
\end{lemma}
\begin{proof}
\begin{enumerate}
\item[(i)] See \cite[p. 39]{K-L}.
\item[(ii)] See \cite[Theorem 6.1 in Chapter 2]{K-L}.
\item[(iii)] See \cite[Lemma 6.2 in Chapter 2]{K-L}.
\end{enumerate}
\end{proof}

\section {Examples of primitive and totally primitive Fricke families}

Let $N\geq2$. In this section, we shall give several examples of primitive and totally primitive Fricke families.

\begin{example}\label{eg1}
Consider the Fricke family $\{g_\mathbf{v}(\tau)^{12N}\}_{\mathbf{v}\in
\mathcal{V}_N}$ consisting of $12N$th powers of Siegel functions.
We want to show that the family is totally primitive.
\par
Suppose that
\begin{equation*}
g_\mathbf{u}(\tau)^{12Nn}=g_\mathbf{v}(\tau)^{12Nn}\quad
\textrm{for some}~\mathbf{u},\,\mathbf{v}\in\mathcal{V}_N~\textrm{and}~n\in\mathbb{N}.
\end{equation*}
Since there is an element $\alpha$ of $\mathrm{GL}_2(\mathbb{Z}/N\mathbb{Z})/\{\pm I_2\}$
such that
\begin{equation*}
\alpha^T\mathbf{u}\equiv\pm\begin{bmatrix}1/N\\0\end{bmatrix}
\Mod{\mathbb{Z}^2},
\end{equation*}
we may assume by (F3) that
\begin{equation}\label{equal}
g_{\left[\begin{smallmatrix}1/N\\0\end{smallmatrix}\right]}(\tau)^{12Nn}=g_\mathbf{v}(\tau)^{12Nn}~\textrm{for}~\mathbf{v}=\begin{bmatrix}v_1\\v_2\end{bmatrix}\in\mathcal{V}_N.
\end{equation}
Applying $\begin{bmatrix}0&1\\-1&0\end{bmatrix}$ to both sides of (\ref{equal}), we attain that
\begin{equation}\label{equal2}
g_{\left[\begin{smallmatrix}0\\1/N\end{smallmatrix}\right]}(\tau)^{12Nn}=
g_{\left[\begin{smallmatrix}-v_2\\v_1\end{smallmatrix}\right]}(\tau)^{12Nn}.
\end{equation}
By Lemma \ref{porder} (i), we obtain from (\ref{equal}) and (\ref{equal2}) that
\begin{equation*}
6Nn\mathbf{B}_2(1/N)=
6Nn\mathbf{B}_2(\langle v_1\rangle)\quad\textrm{and}\quad
6Nn\mathbf{B}_2(0)=
6Nn\mathbf{B}_2(\langle-v_2\rangle),
\end{equation*}
respectively. Thus we deduce by considering the graph of $y=\mathbf{B}_2(x)$ that
\begin{equation*}
v_1\equiv\pm 1/N\Mod{\mathbb{Z}}\quad\textrm{and}\quad
v_2\equiv0\Mod{\mathbb{Z}},
\end{equation*}
and hence $\mathbf{v}\equiv\pm\begin{bmatrix}1/N\\0\end{bmatrix}\Mod{\mathbb{Z}^2}$.
This observation implies that the Fricke family $\{g_\mathbf{v}(\tau)^{12N}\}_{\mathbf{v}\in\mathcal{V}_N}$ is totally primitive.
\end{example}

\begin{example}\label{eg2}
Assume that $N$ is odd and the set
\begin{equation*}
Q_N=
[1,N/2]\cap
\{a\in\mathbb{Z}~|~a\not\equiv\pm1\Mod{N}~\textrm{and}~a^2\equiv\pm1\Mod{N}\}.
\end{equation*}
is nonempty. Let $a\in Q_N$. If we set
\begin{equation}\label{hdef}
h_\mathbf{v}(\tau)=f_\mathbf{v}(\tau)-f_{a\mathbf{v}}(\tau)\quad(\mathbf{v}\in\mathcal{V}_N),
\end{equation}
then we get a Fricke family $\{h_\mathbf{v}(\tau)\}_{\mathbf{v}\in\mathcal{V}_N}$ of level $N$. We want to show that
$\{h_\mathbf{v}(\tau)\}_{\mathbf{v}\in\mathcal{V}_N}$ is primitive, but not
totally primitive.
\par
Suppose that
\begin{equation*}
h_\mathbf{a}(\tau)=h_\mathbf{b}(\tau)\quad\textrm{for some}~\mathbf{a},\,\mathbf{b}\in\mathcal{V}_N.
\end{equation*}
By applying an action of the group $\mathrm{GL}_2(\mathbb{Z}/N\mathbb{Z})/\{\pm I_2\}$,
if necessary, we may assume by (F3) that
\begin{equation}\label{h1b}
h_{\left[\begin{smallmatrix}1/N\\0\end{smallmatrix}\right]}(\tau)=h_\mathbf{b}(\tau)
\quad\textrm{with}~
\mathbf{b}=\begin{bmatrix}b_1\\b_2\end{bmatrix}.
\end{equation}
The action of $\begin{bmatrix}0&1\\-1&0\end{bmatrix}$ on both sides of (\ref{h1b}) yields
\begin{equation}\label{h0b}
h_{\left[\begin{smallmatrix}0\\1/N\end{smallmatrix}\right]}(\tau)=
h_{\left[\begin{smallmatrix}-b_2\\b_1\end{smallmatrix}\right]}(\tau).
\end{equation}
By the definitions (\ref{Fricke}) and (\ref{hdef}), we obtain
from (\ref{h1b}) and (\ref{h0b}) that
\begin{eqnarray*}
\wp_{\left[\begin{smallmatrix}1/N\\0\end{smallmatrix}\right]}(\tau)-
\wp_{\left[\begin{smallmatrix}a/N\\0\end{smallmatrix}\right]}(\tau)&=&
\wp_{\left[\begin{smallmatrix}b_1\\b_2\end{smallmatrix}\right]}(\tau)-
\wp_{\left[\begin{smallmatrix}ab_1\\ab_2\end{smallmatrix}\right]}(\tau),\\
\wp_{\left[\begin{smallmatrix}0\\1/N\end{smallmatrix}\right]}(\tau)-
\wp_{\left[\begin{smallmatrix}0\\a/N\end{smallmatrix}\right]}(\tau)&=&
\wp_{\left[\begin{smallmatrix}-b_2\\b_1\end{smallmatrix}\right]}(\tau)-
\wp_{\left[\begin{smallmatrix}-ab_2\\ab_1\end{smallmatrix}\right]}(\tau).
\end{eqnarray*}
Comparing the $q$-orders by making use of Lemma \ref{porder} (iii), we get
\begin{equation*}
1/N=\min(\langle\pm b_1\rangle,\,\langle\pm ab_1\rangle)
\quad\textrm{and}\quad0=\min(\langle\pm b_2\rangle,\,\langle\pm ab_2\rangle),
\end{equation*}
respectively.
We then deduce from the fact $a^2\equiv\pm1\Mod{N}$ that
\begin{equation*}
b_1\equiv\pm1/N~\textrm{or}~\pm a/N\Mod{\mathbb{Z}}
\quad\textrm{and}\quad b_2\equiv0\Mod{\mathbb{Z}}.
\end{equation*}
If $b_1\equiv\pm a/N\Mod{\mathbb{Z}}$, then we see that
\begin{eqnarray*}
f_{\left[\begin{smallmatrix}1/N\\0\end{smallmatrix}\right]}(\tau)-
f_{\left[\begin{smallmatrix}a/N\\0\end{smallmatrix}\right]}(\tau)
&=&h_{\left[\begin{smallmatrix}1/N\\0\end{smallmatrix}\right]}(\tau)
\quad\textrm{by the definition (\ref{hdef})}\\
&=&h_{\left[\begin{smallmatrix}a/N\\0\end{smallmatrix}\right]}(\tau)\quad
\textrm{by (\ref{h1b}) and (F2)}\\
&=&f_{\left[\begin{smallmatrix}a/N\\0\end{smallmatrix}\right]}(\tau)-
f_{\left[\begin{smallmatrix}a^2/N\\0\end{smallmatrix}\right]}(\tau)
\quad\textrm{by the definition (\ref{hdef})}\\
&=&f_{\left[\begin{smallmatrix}a/N\\0\end{smallmatrix}\right]}(\tau)-
f_{\left[\begin{smallmatrix}1/N\\0\end{smallmatrix}\right]}(\tau)\quad\textrm{by the fact $a^2\equiv\pm1\Mod{N}$ and (F2)},
\end{eqnarray*}
from which it follows that $f_{\left[\begin{smallmatrix}1/N\\0\end{smallmatrix}\right]}(\tau)=
f_{\left[\begin{smallmatrix}a/N\\0\end{smallmatrix}\right]}(\tau)$.
But, this is impossible because $\{f_\mathbf{v}(\tau)\}_{\mathbf{v}\in\mathcal{V}_N}$ is primitive and $a\not\equiv\pm1\Mod{N}$. Thus we attain $\mathbf{b}\equiv\pm
\begin{bmatrix}1/N\\0\end{bmatrix}\Mod{\mathbb{Z}^2}$, which
shows that $\{h_\mathbf{v}(\tau)\}_{\mathbf{v}\in\mathcal{V}_N}$ is primitive.
\par
On the other hand, we derive by the definition (\ref{hdef}), the fact $a^2\equiv
\pm1\Mod{N}$ and (F2) that
\begin{equation*}
h_{a\mathbf{v}}(\tau)=f_{a\mathbf{v}}(\tau)-f_{a^2\mathbf{v}}(\tau)
=f_{a\mathbf{v}}(\tau)-f_\mathbf{v}(\tau)=-h_\mathbf{v}(\tau)\quad(\mathbf{v}\in\mathcal{V}_N),
\end{equation*}
which gives rise to $h_{a\mathbf{v}}(\tau)^2=h_\mathbf{v}(\tau)^2$.
Here we note that $a\mathbf{v}\not\equiv\pm\mathbf{v}\Mod{\mathbb{Z}^2}$
due to the fact $a\not\equiv\pm1\Mod{N}$.
Hence $\{h_\mathbf{v}(\tau)^2\}_{\mathbf{v}\in\mathcal{V}_N}$ is not
primitive.
\par
Therefore, the Fricke family $\{h_\mathbf{v}(\tau)\}_{\mathbf{v}\in\mathcal{V}_N}$ is primitive, whereas not totally primitive.
\end{example}

\begin{example}\label{eg3}
Let $N\geq7$ and $\gcd(6,\,N)=1$.
We claim that the Fricke family
$\{f_\mathbf{v}(\tau)\}_{\mathbf{v}\in\mathcal{V}_N}$ is totally primitive.
\par
Suppose on the contrary that it is not totally primitive.
Then we have
\begin{equation*}
f_\mathbf{a}(\tau)^n=f_\mathbf{b}(\tau)^n
\end{equation*}
for some integer $n\geq2$ and $\mathbf{a},\,\mathbf{b}\in\mathcal{V}_N$ such that $\mathbf{a}
\not\equiv\pm\mathbf{b}\Mod{\mathbb{Z}^2}$. By applying an action of the group $\mathrm{GL}_2(\mathbb{Z}/N\mathbb{Z})/\{\pm I_2\}$, if necessary, we may assume that
\begin{equation}\label{hzh}
f_{\left[\begin{smallmatrix}1/N\\0\end{smallmatrix}\right]}(\tau)=
\zeta f_{\left[\begin{smallmatrix}b_1\\b_2\end{smallmatrix}\right]}(\tau)
\end{equation}
for some $n$th root of unity $\zeta$ and $\begin{bmatrix}b_1\\b_2\end{bmatrix}
\in\mathcal{V}_N$ such that
\begin{equation}\label{bnot1}
\begin{bmatrix}b_1\\
b_2\end{bmatrix}\not\equiv\pm\begin{bmatrix}1/N\\0\end{bmatrix}
\Mod{\mathbb{Z}^2}.
\end{equation}
Through the action of $\begin{bmatrix}1&0\\1&1\end{bmatrix}\in\mathrm{SL}_2(\mathbb{Z}/N
\mathbb{Z})/\{\pm I_2\}$
on both sides of (\ref{hzh}), we get
by (F3) and (A2) that
\begin{equation}\label{hzhh}
f_{\left[\begin{smallmatrix}1/N\\0\end{smallmatrix}\right]}(\tau)
=\zeta f_{\left[\begin{smallmatrix}b_1+b_2\\b_2\end{smallmatrix}\right]}(\tau).
\end{equation}
We see from (\ref{hzh}) and (\ref{hzhh}) that
\begin{equation*}
f_{\left[\begin{smallmatrix}b_1\\b_2\end{smallmatrix}\right]}(\tau)
=f_{\left[\begin{smallmatrix}b_1+b_2\\b_2\end{smallmatrix}\right]}(\tau).
\end{equation*}
Since $\{f_\mathbf{v}(\tau)\}_{\mathbf{v}\in\mathcal{V}_N}$ is primitive, we attain that
\begin{equation*}
\begin{bmatrix}b_1\\b_2\end{bmatrix}\equiv\pm\begin{bmatrix}
b_1+b_2\\b_2\end{bmatrix}\Mod{\mathbb{Z}^2}.
\end{equation*}
If $\begin{bmatrix}b_1\\b_2\end{bmatrix}\equiv-
\begin{bmatrix}b_1+b_2\\b_2\end{bmatrix}\Mod{\mathbb{Z}^2}$, then
we get $2b_1\equiv-b_2\Mod{\mathbb{Z}}$ and $2b_2\equiv0\Mod{\mathbb{Z}}$, and
so $4\begin{bmatrix}b_1\\b_2\end{bmatrix}\equiv\begin{bmatrix}0\\
0\end{bmatrix}\Mod{\mathbb{Z}^2}$. But, this is impossible because
$\begin{bmatrix}b_1\\b_2\end{bmatrix}\in\mathcal{V}_N$ and $N\neq4$.
Thus we must have
\begin{equation*}
\begin{bmatrix}b_1\\b_2\end{bmatrix}\equiv
\begin{bmatrix}b_1+b_2\\b_2\end{bmatrix}\Mod{\mathbb{Z}^2},
\end{equation*}
and hence $b_2\equiv0\Mod{\mathbb{Z}}$.
Write $b_1=a/N$ for an integer $a$ which is relatively prime to $N$ and
$a\not\equiv\pm1\Mod{N}$ by (\ref{bnot1}).
By applying (F2) to the function $f_{\left[\begin{smallmatrix}a/N\\0\end{smallmatrix}\right]}(\tau)$,
we may further assume that
$1<a\leq N/2$.
We then find by (\ref{hzh}) that
\begin{equation*}
\zeta=\frac{f_{\left[\begin{smallmatrix}1/N\\0\end{smallmatrix}\right]}(\tau)}
{f_{\left[\begin{smallmatrix}a/N\\0\end{smallmatrix}\right]}(\tau)},
\end{equation*}
and hence we obtain by acting $\begin{bmatrix}1&0\\0&-1\end{bmatrix}\in G_N$ that
\begin{equation*}
\zeta^{-1}=\frac{f_{\left[\begin{smallmatrix}1/N\\0\end{smallmatrix}\right]}(\tau)}
{f_{\left[\begin{smallmatrix}a/N\\0\end{smallmatrix}\right]}(\tau)}
\end{equation*}
due to (A1) and (F3). Since $\{f_\mathbf{v}(\tau)\}_{\mathbf{v}\in\mathcal{V}_N}$
is primitive and $a\not\equiv\pm1\Mod{N}$, we conclude $\zeta=-1$, and so
\begin{equation}\label{h1-ha}
f_{\left[\begin{smallmatrix}1/N\\0\end{smallmatrix}\right]}(\tau)=-
f_{\left[\begin{smallmatrix}a/N\\0\end{smallmatrix}\right]}(\tau).
\end{equation}
The action of $\begin{bmatrix}a&0\\0&a\end{bmatrix}$ on both sides of
(\ref{h1-ha}) yields
\begin{equation}\label{ha-ha2}
f_{\left[\begin{smallmatrix}a/N\\0\end{smallmatrix}\right]}(\tau)=-
f_{\left[\begin{smallmatrix}a^2/N\\0\end{smallmatrix}\right]}(\tau)
\end{equation}
by (F3). It then follows from (\ref{h1-ha}) and (\ref{ha-ha2}) that
\begin{equation*}
f_{\left[\begin{smallmatrix}1/N\\0\end{smallmatrix}\right]}(\tau)=
f_{\left[\begin{smallmatrix}a^2/N\\0\end{smallmatrix}\right]}(\tau),
\end{equation*}
which implies that
\begin{equation}\label{asquare}
a^2\equiv\pm1\Mod{N}
\end{equation}
because $\{f_\mathbf{v}(\tau)\}_{\mathbf{v}\in\mathcal{V}_N}$ is primitive.
Acting $\begin{bmatrix}2&0\\0&2\end{bmatrix}\in G_N$ on both sides of
(\ref{h1-ha}), we also obtain by (F3) that
\begin{equation}\label{h2-h2a}
f_{\left[\begin{smallmatrix}2/N\\0\end{smallmatrix}\right]}(\tau)=
-f_{\left[\begin{smallmatrix}2a/N\\0\end{smallmatrix}\right]}(\tau).
\end{equation}
We then derive by the definition (\ref{Fricke}),
(\ref{h1-ha}) and (\ref{h2-h2a}) that
\begin{equation*}
\wp_{\left[\begin{smallmatrix}1/N\\0\end{smallmatrix}\right]}(\tau)-
\wp_{\left[\begin{smallmatrix}2/N\\0\end{smallmatrix}\right]}(\tau)
=-\wp_{\left[\begin{smallmatrix}a/N\\0\end{smallmatrix}\right]}(\tau)+
\wp_{\left[\begin{smallmatrix}2a/N\\0\end{smallmatrix}\right]}(\tau).
\end{equation*}
By Lemma \ref{porder} (iii) and the fact $1\leq a\leq N/2$, we achieve that
\begin{eqnarray*}
\mathrm{ord}_q~(\wp_{\left[\begin{smallmatrix}1/N\\0\end{smallmatrix}\right]}(\tau)-
\wp_{\left[\begin{smallmatrix}2/N\\0\end{smallmatrix}\right]}(\tau))&=&
\min(\langle\pm1/N\rangle,\,\langle\pm2/N\rangle)\\
&=&1/N\\
&=&\mathrm{ord}_q~(-\wp_{\left[\begin{smallmatrix}a/N\\0\end{smallmatrix}\right]}(\tau)+
\wp_{\left[\begin{smallmatrix}2a/N\\0\end{smallmatrix}\right]}(\tau))\\
&=&\min(\langle\pm a/N\rangle,\,\langle\pm 2a/N\rangle)\\
&=&\min(\min(a/N,\,(N-a)/N),\,\min(2a/N,\,(N-2a)/N))\\
&=&\left\{
\begin{array}{ll}
a/N & \textrm{if}~1\leq a\leq N/3,\\
(N-2a)/N& \textrm{if}~N/3<a\leq N/2.
\end{array}\right.
\end{eqnarray*}
Moreover, since $a\neq1$, we
must get $1/N=(N-2a)/N$, and so $a=(N-1)/2$. We then obtain from (\ref{asquare})
that
\begin{equation*}
(N^2-2N+1)/4\equiv\pm1\Mod{N}.
\end{equation*}
But, this contradicts the assumption $N\geq7$.
\par
Therefore, we conclude that the primitive Fricke family $\{f_\mathbf{v}(\tau)\}_{\mathbf{v}\in\mathcal{V}_N}$ is also totally primitive.
\end{example}

\section {Generators of function fields}

Let $N\geq2$. As an application of (totally) primitive Fricke families,
we shall construct primitive generators of $\mathbb{C}(X(N))$ and $\mathcal{F}_N$
over $\mathbb{C}(X(1))=\mathbb{C}(j(\tau))$
and $\mathcal{F}_1=\mathbb{Q}(j(\tau))$, respectively.

\begin{proposition}\label{three}
Let $\{h_\mathbf{v}(\tau)\}_{\mathbf{v}\in\mathcal{V}_N}$ be a
primitive Fricke family of level $N$. Then we have
\begin{equation*}
\mathbb{C}(X(N))=\mathbb{C}\left(j(\tau),h_{\left[\begin{smallmatrix}1/N\\0\end{smallmatrix}\right]}(\tau),\,
h_{\left[\begin{smallmatrix}0\\1/N\end{smallmatrix}\right]}(\tau)\right).
\end{equation*}
\end{proposition}
\begin{proof}
Recall that $\mathbb{C}(X(N))$ is a Galois extension of $\mathbb{C}(X(1))$ with
\begin{equation*}
\mathrm{Gal}(\mathbb{C}(X(N))/\mathbb{C}(X(1)))\simeq\mathrm{SL}_2(\mathbb{Z})/\pm\Gamma(N).
\end{equation*}
Let $\gamma=\begin{bmatrix}a&b\\c&d\end{bmatrix}$ be an element of $\mathrm{SL}_2(\mathbb{Z})$ which leaves both $h_{\left[\begin{smallmatrix}1/N\\0\end{smallmatrix}
\right]}(\tau)$ and $h_{\left[\begin{smallmatrix}0\\1/N\end{smallmatrix}\right]}(\tau)$ fixed. We
then see by (F3) that
\begin{equation*}
h_{\left[\begin{smallmatrix}1/N\\0\end{smallmatrix}\right]}(\tau)=
h_{\left[\begin{smallmatrix}1/N\\0\end{smallmatrix}\right]}(\tau)^\gamma=
h_{\left[\begin{smallmatrix}a/N\\b/N\end{smallmatrix}\right]}(\tau)
\quad\textrm{and}\quad
h_{\left[\begin{smallmatrix}0\\1/N\end{smallmatrix}\right]}(\tau)=
h_{\left[\begin{smallmatrix}0\\1/N\end{smallmatrix}\right]}(\tau)^\gamma=
h_{\left[\begin{smallmatrix}c/N\\d/N\end{smallmatrix}\right]}(\tau).
\end{equation*}
Now that $\{h_\mathbf{v}(\tau)\}_{\mathbf{v}\in\mathcal{V}_N}$ is primitive, we get
\begin{equation*}
\begin{bmatrix}a/N\\b/N\end{bmatrix}\equiv\pm
\begin{bmatrix}1/N\\0\end{bmatrix}\quad\textrm{and}\quad
\begin{bmatrix}c/N\\d/N\end{bmatrix}\equiv\pm
\begin{bmatrix}0\\1/N\end{bmatrix}\Mod{\mathbb{Z}^2}.
\end{equation*}
Moreover, we deduce from $\det(\gamma)=ad-bc=1$ that
$a\equiv d\equiv\pm1$; and hence
$\gamma\equiv\pm I_2\Mod{N}$ and so
$\gamma\in\pm\Gamma(N)$.
\par
This yields by Galois theory that
$h_{\left[\begin{smallmatrix}1/N\\0
\end{smallmatrix}\right]}(\tau)$ and $h_{\left[\begin{smallmatrix}0\\1/N\end{smallmatrix}\right]}(\tau)$
generate $\mathbb{C}(X(N))$ over $\mathbb{C}(X(1))=\mathbb{C}(j(\tau))$.
\end{proof}

\begin{example}
Since $\{g_\mathbf{v}(\tau)^{12N}\}_{\mathbf{v}\in\mathcal{V}_N}$ is totally primitive by  Example \ref{eg1},
$\{g_\mathbf{v}(\tau)^{12Nn}\}_{\mathbf{v}\in\mathcal{V}_N}$
is primitive for each positive integer $n$.
By applying Proposition \ref{three} to each family
$\{g_\mathbf{v}(\tau)^{12Nn}\}_{\mathbf{v}\in\mathcal{V}_N}$ and
using the fact that $\mathbb{C}(X(N))$ is a field,
we obtain
\begin{equation*}
\mathbb{C}(X(N))=\mathbb{C}\left(j(\tau),\,g_{\left[\begin{smallmatrix}
1/N\\0\end{smallmatrix}\right]}(\tau)^{12Nn},\,
g_{\left[\begin{smallmatrix}
0\\1/N\end{smallmatrix}\right]}(\tau)^{12Nn}\right)
\end{equation*}
for any nonzero integer $n$.
\end{example}

\begin{theorem}\label{usingFricke}
We have
\begin{itemize}
\item[\textup{(i)}] $\mathbb{C}(X(N))= \mathbb{C}\left(j(\tau),\,f_{\left[\begin{smallmatrix}1/N\\0\end{smallmatrix}\right]}(\tau)
-f_{\left[\begin{smallmatrix}0\\1/N\end{smallmatrix}\right]}(\tau)^{-1}\right)$.
\item[\textup{(ii)}] $\mathcal{F}_N=
\mathbb{Q}\left(j(\tau),\,\zeta_N\left(f_{\left[\begin{smallmatrix}1/N\\0\end{smallmatrix}\right]}(\tau)
-f_{\left[\begin{smallmatrix}0\\1/N\end{smallmatrix}\right]}(\tau)^{-1}\right)\right)$.
\end{itemize}
\end{theorem}
\begin{proof}
\begin{enumerate}
\item[(i)]
Let $\gamma=\begin{bmatrix}a&b\\c&d\end{bmatrix}$ be an element of $\mathrm{SL}_2(\mathbb{Z})$ which leaves $f_{\left[\begin{smallmatrix}1/N\\0\end{smallmatrix}\right]}(\tau)-
f_{\left[\begin{smallmatrix}0\\1/N\end{smallmatrix}\right]}(\tau)^{-1}$ fixed.
By (F3) we attain that
\begin{equation*}
f_{\left[\begin{smallmatrix}1/N\\0\end{smallmatrix}\right]}(\tau)-
f_{\left[\begin{smallmatrix}0\\1/N\end{smallmatrix}\right]}(\tau)^{-1}
=(f_{\left[\begin{smallmatrix}1/N\\0\end{smallmatrix}\right]}(\tau)-
f_{\left[\begin{smallmatrix}0\\1/N\end{smallmatrix}\right]}(\tau)^{-1})^\gamma\\
=f_{\left[\begin{smallmatrix}a/N\\b/N\end{smallmatrix}\right]}(\tau)-
f_{\left[\begin{smallmatrix}c/N\\d/N\end{smallmatrix}\right]}(\tau)^{-1},
\end{equation*}
from which it follows that
\begin{equation}\label{difffrac}
f_{\left[\begin{smallmatrix}1/N\\0\end{smallmatrix}\right]}(\tau)-
f_{\left[\begin{smallmatrix}a/N\\b/N\end{smallmatrix}\right]}(\tau)
=f_{\left[\begin{smallmatrix}0\\1/N\end{smallmatrix}\right]}(\tau)^{-1}-
f_{\left[\begin{smallmatrix}c/N\\d/N\end{smallmatrix}\right]}(\tau)^{-1}
=\frac{f_{\left[\begin{smallmatrix}c/N\\d/N\end{smallmatrix}\right]}(\tau)-
f_{\left[\begin{smallmatrix}0\\1/N\end{smallmatrix}\right]}(\tau)}
{f_{\left[\begin{smallmatrix}0\\1/N\end{smallmatrix}\right]}(\tau)
f_{\left[\begin{smallmatrix}c/N\\d/N\end{smallmatrix}\right]}(\tau)}.
\end{equation}
If $\begin{bmatrix}a/N\\c/N\end{bmatrix}\not
\equiv\pm\begin{bmatrix}1/N\\0\end{bmatrix}\Mod{\mathbb{Z}^2}$, then
we deduce
\begin{equation*}
f_{\left[\begin{smallmatrix}1/N\\0\end{smallmatrix}\right]}(\tau)
-f_{\left[\begin{smallmatrix}a/N\\b/N\end{smallmatrix}\right]}(\tau)\neq0
\end{equation*}
due to the fact that $\{f_\mathbf{v}(\tau)\}_{\mathbf{v}\in\mathcal{V}_N}$ is primitive.
We then see from (\ref{difffrac}) that
\begin{equation*}
f_{\left[\begin{smallmatrix}c/N\\d/N\end{smallmatrix}\right]}(\tau)
-f_{\left[\begin{smallmatrix}0\\1/N\end{smallmatrix}\right]}(\tau)\neq0,
\end{equation*}
which yields $\begin{bmatrix}c/N\\d/N\end{bmatrix}\not
\equiv\pm\begin{bmatrix}0\\1/N\end{bmatrix}\Mod{\mathbb{Z}^2}$.
Now, consider the relation
\begin{equation}\label{Frickeunit}
f_{\left[\begin{smallmatrix}0\\1/N\end{smallmatrix}\right]}(\tau)
f_{\left[\begin{smallmatrix}c/N\\d/N\end{smallmatrix}\right]}(\tau)=
\frac{f_{\left[\begin{smallmatrix}c/N\\d/N\end{smallmatrix}\right]}(\tau)-
f_{\left[\begin{smallmatrix}0\\1/N\end{smallmatrix}\right]}(\tau)}
{f_{\left[\begin{smallmatrix}1/N\\0\end{smallmatrix}\right]}(\tau)
-f_{\left[\begin{smallmatrix}a/N\\b/N\end{smallmatrix}\right]}(\tau)}
\end{equation}
derived from (\ref{difffrac}). Since $g_2(\zeta_3)=0$ (\cite[p. 37]{Lang}), the left side of (\ref{Frickeunit}) vanishes at $\zeta_3$ by the definition (\ref{Fricke}), whereas
the right side of (\ref{Frickeunit}) has neither zeros nor poles on $\mathbb{H}$
by Lemma \ref{porder} (ii). This gives a contradiction.
Thus we achieve that
\begin{equation*}
\begin{bmatrix}a/N\\b/N\end{bmatrix}\equiv\pm
\begin{bmatrix}1/N\\0\end{bmatrix}\Mod{\mathbb{Z}^2},
\end{equation*}
and get by (\ref{difffrac}) that
\begin{equation*}
\begin{bmatrix}c/N\\d/N\end{bmatrix}\equiv\pm
\begin{bmatrix}0\\1/N\end{bmatrix}\Mod{\mathbb{Z}^2}.
\end{equation*}
Furthermore, we obtain by the fact $\det(\gamma)=ad-bc=1$ that
$a\equiv d\equiv\pm 1\Mod{N}$ and $b\equiv c\equiv0\Mod{N}$; and hence
$\gamma\in\pm\Gamma(N)$. This proves (i) by Galois theory.
\item[(ii)] We get by (i) and \cite[Theorem 6.6]{Shimura} that
\begin{equation*}
\mathcal{F}_N=
\mathbb{Q}\left(\zeta_N,\,j(\tau),\,f_{\left[\begin{smallmatrix}1/N\\0\end{smallmatrix}\right]}(\tau)
-f_{\left[\begin{smallmatrix}0\\1/N\end{smallmatrix}\right]}(\tau)^{-1}\right).
\end{equation*}
Thus, if we set
\begin{equation*}
F=\mathbb{Q}\left(j(\tau),\,\zeta_N\left(f_{\left[\begin{smallmatrix}1/N\\0\end{smallmatrix}\right]}(\tau)
-f_{\left[\begin{smallmatrix}0\\1/N\end{smallmatrix}\right]}(\tau)^{-1}\right)\right),
\end{equation*}
then $\mathrm{Gal}(\mathcal{F}_N/F)$ is a subgroup of
\begin{equation*}
G_N=\left\{\begin{bmatrix}1&0\\0&d\end{bmatrix}~|~d\in(\mathbb{Z}/N\mathbb{Z})^\times
\right\}.
\end{equation*}
Let $\gamma=\begin{bmatrix}1&0\\0&d\end{bmatrix}$ be an element of $G_N$
which fixes the field $F$ elementwise. We then see by (F3) and (A1) that
\begin{eqnarray*}
\zeta_N\left(f_{\left[\begin{smallmatrix}1/N\\0\end{smallmatrix}\right]}(\tau)
-f_{\left[\begin{smallmatrix}0\\1/N\end{smallmatrix}\right]}(\tau)^{-1}\right)&=&
\left(\zeta_N\left(f_{\left[\begin{smallmatrix}1/N\\0\end{smallmatrix}\right]}(\tau)
-f_{\left[\begin{smallmatrix}0\\1/N\end{smallmatrix}\right]}(\tau)^{-1}\right)\right)^\gamma\\
&=&\zeta_N^d\left(f_{\left[\begin{smallmatrix}1/N\\0\end{smallmatrix}\right]}(\tau)
-f_{\left[\begin{smallmatrix}0\\d/N\end{smallmatrix}\right]}(\tau)^{-1}\right),
\end{eqnarray*}
from which it follows that
\begin{equation}\label{zetafrac}
\zeta_N^{d-1}=
\frac{f_{\left[\begin{smallmatrix}1/N\\0\end{smallmatrix}\right]}(\tau)-
f_{\left[\begin{smallmatrix}0\\1/N\end{smallmatrix}\right]}(\tau)^{-1}}
{f_{\left[\begin{smallmatrix}1/N\\0\end{smallmatrix}\right]}(\tau)-
f_{\left[\begin{smallmatrix}0\\d/N\end{smallmatrix}\right]}(\tau)^{-1}}.
\end{equation}
Here, we note by (F2) and (F3) that the right side of (\ref{zetafrac}) is fixed by the action of $\begin{bmatrix}1&0\\0&-1
\end{bmatrix}$. Thus we attain by (A1) that
\begin{equation*}
\zeta_N^{d-1}=(\zeta_N^{d-1})^{\left[\begin{smallmatrix}1&0\\0&-1\end{smallmatrix}\right]}
=\zeta_N^{-(d-1)},
\end{equation*}
and so $\zeta_N^{d-1}=\pm1$.
If $\zeta_N^{d-1}=-1$, then we derive by (\ref{zetafrac}) that
\begin{equation}\label{2fff}
2f_{\left[\begin{smallmatrix}1/N\\0\end{smallmatrix}\right]}(\tau)
=f_{\left[\begin{smallmatrix}0\\1/N\end{smallmatrix}\right]}(\tau)^{-1}
+f_{\left[\begin{smallmatrix}0\\d/N\end{smallmatrix}\right]}(\tau)^{-1}.
\end{equation}
Due to (F3), the right side of (\ref{2fff}) is fixed by the action of $\begin{bmatrix}1&1\\0&1\end{bmatrix}$,
but we see that
\begin{equation*}
2f_{\left[\begin{smallmatrix}1/N\\0\end{smallmatrix}\right]}(\tau)^{\left[\begin{smallmatrix}
1&1\\0&1\end{smallmatrix}\right]}
=2f_{\left[\begin{smallmatrix}1/N\\1/N\end{smallmatrix}\right]}(\tau)\neq
2f_{\left[\begin{smallmatrix}1/N\\0\end{smallmatrix}\right]}(\tau)
\end{equation*}
because $\{f_\mathbf{v}(\tau)\}_{\mathbf{v}\in\mathcal{V}_N}$ is primitive.
This yields a contradiction.
Therefore, we must have $\zeta_N^{d-1}=1$; and hence $d\equiv 1\Mod{N}$
and $\gamma=\begin{bmatrix}1&0\\0&1\end{bmatrix}$.
This implies by Galois theory that $F=\mathcal{F}_N$, as desired.
\end{enumerate}
\end{proof}

By using the idea of Example \ref{eg1}, we further establish the following theorem.

\begin{theorem}\label{usingSiegel}
Let $n$ be any nonzero integer.
\begin{itemize}
\item[\textup{(i)}] $\mathbb{C}(X(N))=
\mathbb{C}\left(j(\tau),\,
g_{\left[\begin{smallmatrix}1/N\\0\end{smallmatrix}\right]}(\tau)^{12Nn}
g_{\left[\begin{smallmatrix}0\\1/N\end{smallmatrix}\right]}(\tau)^{24Nn}\right)$.
\item[\textup{(ii)}]
$\mathcal{F}_N=
\mathbb{Q}\left(j(\tau),\,
\zeta_Ng_{\left[\begin{smallmatrix}1/N\\0\end{smallmatrix}\right]}(\tau)^{12Nn}
g_{\left[\begin{smallmatrix}0\\1/N\end{smallmatrix}\right]}(\tau)^{24Nn}\right)$.
\end{itemize}
\end{theorem}
\begin{proof}
\begin{enumerate}
\item[(i)] Let
\begin{equation*}
g(\tau)=g_{\left[\begin{smallmatrix}1/N\\0\end{smallmatrix}\right]}(\tau)^{12Nn}
g_{\left[\begin{smallmatrix}0\\1/N\end{smallmatrix}\right]}(\tau)^{24Nn},
\end{equation*}
which belongs to $\mathbb{C}(X(N))$.
And, let $\gamma=\begin{bmatrix}
a&b\\c&d\end{bmatrix}$ be an element of $\mathrm{SL}_2(\mathbb{Z})$
leaving $g(\tau)$ fixed.
We get by Lemma \ref{porder} (i) and (F3) that
\begin{eqnarray*}
\mathrm{ord}_q~g(\tau)
&=&
6Nn\mathbf{B}_2(1/N)+12Nn\mathbf{B}_2(0)\\
&=&\mathrm{ord}_q~g(\tau)^\gamma\\
&=&\mathrm{ord}_q~g_{\left[\begin{smallmatrix}a/N\\b/N\end{smallmatrix}\right]}(\tau)^{12Nn}
g_{\left[\begin{smallmatrix}c/N\\d/N\end{smallmatrix}\right]}(\tau)^{24Nn}\\
&=&6Nn\mathbf{B}_2(\langle a/N\rangle)+12Nn\mathbf{B}_2(\langle c/N\rangle).
\end{eqnarray*}
By considering the shape of the graph $y=\mathbf{B}_2(x)$ on the domain $[0,1)$, we deduce that
\begin{equation*}
\langle c/N\rangle=0\quad\textrm{and}\quad
\langle a/N\rangle=1/N~\textrm{or}~(N-1)/N,
\end{equation*}
and so $c\equiv0\pmod{N}$ and $a\equiv\pm1\pmod{N}$.
Moreover, we achieve by the fact
$\det(\gamma)=ad-bc=1$ that
$a\equiv d\equiv\pm 1\pmod{N}$.
On the other hand, we derive by (F3) and Lemma \ref{porder} (i) that
\begin{eqnarray*}
\mathrm{ord}_q~g(\tau)^{\left[\begin{smallmatrix}0&-1\\1&0\end{smallmatrix}\right]}&=&
\mathrm{ord}_q~
g_{\left[\begin{smallmatrix}0\\-1/N\end{smallmatrix}\right]}(\tau)^{12Nn}
g_{\left[\begin{smallmatrix}1/N\\0\end{smallmatrix}\right]}(\tau)^{24Nn}
\\
&=&
6Nn\mathbf{B}_2(0)+12Nn\mathbf{B}_2(1/N)\\
&=&\mathrm{ord}_q~(g(\tau)^{\gamma})^{
\left[\begin{smallmatrix}0&-1\\1&0\end{smallmatrix}\right]}\\
&=&\mathrm{ord}_q~g(\tau)^{\left[\begin{smallmatrix}b&-a\\d&-c\end{smallmatrix}\right]}\\
&=&\mathrm{ord}_q~
g_{\left[\begin{smallmatrix}b/N\\-a/N\end{smallmatrix}\right]}(\tau)^{12Nn}
g_{\left[\begin{smallmatrix}d/N\\-c/N\end{smallmatrix}\right]}(\tau)^{24Nn}\\
&=&6Nn\mathbf{B}_2(\langle b/N\rangle)+12Nn\mathbf{B}_2(\langle d/N\rangle),
\end{eqnarray*}
from which we conclude $b\equiv0\pmod{N}$. Hence $\gamma$ belongs to $\pm
\Gamma(N)$, which proves that $g(\tau)$ generates the field $\mathbb{C}(X(N))$ over $\mathbb{C}(X(1))$.
\item[(ii)] Let
\begin{equation*}
F=\mathbb{Q}\left(j(\tau),\,\zeta_Ng_{\left[\begin{smallmatrix}1/N\\0\end{smallmatrix}\right]}(\tau)^{12Nn}
g_{\left[\begin{smallmatrix}0\\1/N\end{smallmatrix}\right]}(\tau)^{24Nn}\right).
\end{equation*}
Then, $F$ is a subfield of $\mathcal{F}_N$ and $\mathrm{Gal}(\mathcal{F}_N/F)$ is a subgroup of $G_N$. Let $\alpha=\begin{bmatrix}1&0\\0&d\end{bmatrix}$ be an element
of $G_N$ which leaves $\zeta_N g(\tau)$ fixed. Letting $\beta=\begin{bmatrix}0&-1\\1&0\end{bmatrix}
\in\mathrm{SL}_2(\mathbb{Z}/N\mathbb{Z})/\{\pm I_2\}$, we deduce by Lemma \ref{porder} (i), (F3) and (A1) that
\begin{eqnarray*}
\mathrm{ord}_q~(\zeta_Ng(\tau))^\beta&=&
\mathrm{ord}_q~\zeta_Ng_{\left[\begin{smallmatrix}0\\-1/N\end{smallmatrix}\right]}(\tau)^{12Nn}
g_{\left[\begin{smallmatrix}1/N\\0\end{smallmatrix}\right]}(\tau)^{24Nn}\\
&=&
6Nn\mathbf{B}_2(0)+12Nn\mathbf{B}_2(1/N)\\
&=&\mathrm{ord}_q~((\zeta_Ng(\tau))^\alpha)^\beta\\
&=&\mathrm{ord}_q~(\zeta_N^d
g_{\left[\begin{smallmatrix}1/N\\0\end{smallmatrix}\right]}(\tau)^{12Nn}
g_{\left[\begin{smallmatrix}0\\d/N\end{smallmatrix}\right]}(\tau)^{24Nn})^\beta\\
&=&\mathrm{ord}_q~\zeta_N^dg_{\left[\begin{smallmatrix}0\\-1/N\end{smallmatrix}\right]}(\tau)^{12Nn}
g_{\left[\begin{smallmatrix}d/N\\0\end{smallmatrix}\right]}(\tau)^{24Nn}\\
&=&6Nn\mathbf{B}_2(0)+12Nn\mathbf{B}_2(\langle d/N\rangle).
\end{eqnarray*}
Thus we obtain $d\equiv\pm1\Mod{N}$. It is clear that if $N=2$, then $d\equiv1\Mod{N}$.
If $N\geq3$ and $d\equiv-1\Mod{N}$, then we get by (A1), (F2) and (F3) that
\begin{equation*}
\zeta_Ng(\tau)=\zeta_Ng_{\left[\begin{smallmatrix}1/N\\0\end{smallmatrix}\right]}(\tau)^{12Nn}
g_{\left[\begin{smallmatrix}0\\1/N\end{smallmatrix}\right]}(\tau)^{24Nn}
=(\zeta_Ng(\tau))^\alpha
=\zeta_N^{-1}g_{\left[\begin{smallmatrix}1/N\\0\end{smallmatrix}\right]}(\tau)^{12Nn}
g_{\left[\begin{smallmatrix}0\\1/N\end{smallmatrix}\right]}(\tau)^{24Nn},
\end{equation*}
and so $\zeta_N^2=1$. But, this is impossible.
Therefore, we always have $d\equiv1\Mod{N}$, from which $F=\mathcal{F}_N$ by Galois theory.
\end{enumerate}
\end{proof}

\begin{remark}\label{singularmodel}
It is well known that $g_\mathbf{v}(\tau)^{12N}$ are integral over $\mathbb{Z}[j(\tau)]$
for all $\mathbf{v}\in\mathcal{V}_N$ (\cite[$\S$3]{K-S}).
Thus, if $n>0$ and $g(\tau)=g_{\left[\begin{smallmatrix}1/N\\0\end{smallmatrix}\right]}(\tau)^{12Nn}
g_{\left[\begin{smallmatrix}0\\1/N\end{smallmatrix}\right]}(\tau)^{24Nn}$, then
there is a polynomial $f_N(x,\,y)\in\mathbb{Z}[x,\,y]$ for which
$f_N(x,\,y)$ is monic in $x$ and $f_N(g(\tau),\,j(\tau))=0$. That is, the equation $f_N(x,\,y)=0$ gives
rise to an affine singular model of the modular curve $X(N)$ over $\mathbb{Q}$.
For example, if $N=2$ and $n=1$, then one can estimate
\begin{eqnarray*}
f_N(x,\,y)&=&x^6+(-2y^3+2^8\cdot3^2y^2+2^{18}\cdot3y-2^{25}\cdot3)x^5\\
&&+(y^6-2^9\cdot3^2y^5+2^{16}\cdot3^2\cdot13y^4
-2^{25}\cdot163y^3+2^{36}\cdot3^3y^2
-2^{44}\cdot3y
+2^{48}\cdot3\cdot5)x^4\\
&&+(-2^{25}y^6+2^{40}\cdot3\cdot67y^4-2^{55}\cdot7y^3
+2^{57}\cdot3^2\cdot47y^2+2^{67}\cdot3^2y
-2^{74}\cdot5)x^3\\
&&+(2^{48}y^6-2^{57}\cdot3^2y^5
+2^{64}\cdot3^2\cdot13y^4
-2^{73}\cdot163y^3+2^{84}\cdot3^3y^2
-2^{92}\cdot3y+2^{96}\cdot3\cdot5)x^2\\
&&+(-2^{97}y^3+2^{104}\cdot3^2y^2
+2^{114}\cdot3y-2^{121}\cdot3)x+2^{144}
\end{eqnarray*}
by using the Fourier expansions of
Siegel functions and $j(\tau)$ (see \cite[p. 29]{K-L} and \cite[Theorem 12.17]{Cox}).
\end{remark}

\section {Application to class fields}

Let $K$ be an imaginary quadratic field and $N\geq2$.
As a consequence of the main theorem of complex multiplication, we obtain that
$H_K=K(j(\tau_K))$ and
\begin{equation}\label{CM}
K_{(N)}=K(h(\tau_K)~|~h(\tau)\in\mathcal{F}_N~\textrm{is finite at}~\tau_K)
\end{equation}
(\cite[Theorem 1 and Corollary to Theorem 2 in Chapter 10]{Lang}).
Let
\begin{equation*}
\min(\tau_K,\,\mathbb{Q})=x^2+B_Kx+C_K\in\mathbb{Z}[x],
\end{equation*}
and define a subgroup $W_{K,\,N}$ of $\mathrm{GL}_2(\mathbb{Z}/N\mathbb{Z})$ by
\begin{eqnarray*}
W_{K,\,N}=\left\{\gamma=\begin{bmatrix}t-B_Ks & -C_Ks\\s&t\end{bmatrix}~|~
t,\,s\in\mathbb{Z}/N\mathbb{Z}~\textrm{such that}~\gamma\in\mathrm{GL}_2(\mathbb{Z}/N\mathbb{Z})\right\}.
\end{eqnarray*}
If $K$ is different from $\mathbb{Q}(\sqrt{-1})$ and
$\mathbb{Q}(\sqrt{-3})$, then by the Shimura reciprocity law we have the isomorphism
\begin{eqnarray}
W_{K,\,N}/\{\pm I_2\}&\stackrel{\sim}{\longrightarrow}&\mathrm{Gal}(K_{(N)}/H_K)\label{reciprocity}\\
\gamma&\mapsto&(h(\tau_K)\mapsto h^\gamma(\tau_K)~|~
h(\tau)\in\mathcal{F}_N~\textrm{is finite at}~\tau_K)\nonumber
\end{eqnarray}
(\cite[$\S$3]{Stevenhagen}).

\begin{lemma}\label{rootofunity}
If $m$ is a positive integer such that $\zeta_m\in K_{(N)}$, then
$m$ divides $12N$.
\end{lemma}
\begin{proof}
See \cite[Lemma 4.3 (i) in Chapter 9]{K-L}.
\end{proof}

Let $\{h_\mathbf{v}(\tau)\}_{\mathbf{v}\in\mathcal{V}_N}$ be
a totally primitive Fricke family of level $N$, and let
$d_N(\tau)$ be the discriminant of $h_{\left[\begin{smallmatrix}0\\1/N\end{smallmatrix}\right]}(\tau)^{12N}$ over $\mathcal{F}_1$.
Define an equivalence relation $\sim$ on the set $\mathcal{V}_N$ by
\begin{equation*}
\mathbf{u}~\sim~\mathbf{v}\quad\Longleftrightarrow\quad
\mathbf{u}\equiv\pm\mathbf{v}\Mod{\mathbb{Z}^2}.
\end{equation*}
Since $\mathrm{Gal}(\mathcal{F}_N/\mathcal{F}_1)\simeq\mathrm{GL}_2(\mathbb{Z}/N\mathbb{Z})/\{\pm I_2\}$ and $\{h_\mathbf{v}(\tau)^{12N}\}_{\mathbf{v}\in\mathcal{V}_N}$ is primitive, we get by (F3) that
\begin{equation*}
d_N(\tau)=\pm\prod_{\begin{smallmatrix}[\mathbf{u}],\,[\mathbf{v}]\in\mathcal{V}_N/\sim\\
\textrm{such that}~[\mathbf{u}]\neq[\mathbf{v}]\end{smallmatrix}}\left(h_\mathbf{u}(\tau)^{12N}-h_\mathbf{v}(\tau)^{12N}\right),
\end{equation*}
where $[\mathbf{u}]$ and $[\mathbf{v}]$ stand for the equivalences classes of $\mathbf{u}$ and $\mathbf{v}$ in $\mathcal{V}_N$, respectively.
Note that $d_N(\tau)$ is a nonzero element of $\mathcal{F}_1$ which is weakly holomorphic.
Thus it has only finitely many zeros on the modular curve $X(1)$, and hence
the set
\begin{equation*}
S_N=\{\textrm{imaginary quadratic fields}~K~|~d_N(\tau_K)=0\}\cup\{\mathbb{Q}(\sqrt{-1}),\mathbb{Q}(\sqrt{-3})\}
\end{equation*}
is finite.

\begin{theorem}\label{invariant}
Let $K$ be an imaginary quadratic field lying outside the set $S_N$, and let
$\{h_\mathbf{v}(\tau)\}_{\mathbf{v}\in\mathcal{V}_N}$ be a totally primitive Fricke family of level $N$. If $h_{\left[\begin{smallmatrix}0\\1/N\end{smallmatrix}\right]}(\tau_K)$ is nonzero, then
\begin{equation*}
h_{\left[\begin{smallmatrix}0\\1/N\end{smallmatrix}\right]}(\tau_K)^n\quad
\end{equation*}
generates $K_{(N)}$ over $H_K$ for any nonzero integer $n$.
\end{theorem}
\begin{proof}
Suppose on the contrary that
$h_{\left[\begin{smallmatrix}0\\1/N\end{smallmatrix}\right]}(\tau_K)^n$ does not generates $K_{(N)}$ over $H_K$ for some nonzero integer $n$. Then, there is a nonidentity element
$\alpha$ of $\mathrm{Gal}(K_{(N)}/H_K)$ leaving $h_{\left[\begin{smallmatrix}0\\1/N\end{smallmatrix}\right]}(\tau_K)^n$ fixed.
Due to the isomorphism given in (\ref{reciprocity}), the Galois element $\alpha$ corresponds to
a matrix $\begin{bmatrix}t-B_Ks&-C_Ks\\s&t\end{bmatrix}$ in $W_{K,\,N}/\{\pm I_2\}$
with $\begin{bmatrix}s\\t\end{bmatrix}\neq\pm
\begin{bmatrix}0\\1\end{bmatrix}\Mod{N}$. We then achieve that
\begin{eqnarray*}
h_{\left[\begin{smallmatrix}0\\1/N\end{smallmatrix}\right]}(\tau_K)^n&=&
(h_{\left[\begin{smallmatrix}0\\1/N\end{smallmatrix}\right]}(\tau_K)^n)^\alpha\\
&=&(h_{\left[\begin{smallmatrix}0\\1/N\end{smallmatrix}\right]}(\tau_K)^\alpha)^n\\
&=&h_{\left[\begin{smallmatrix}t-B_Ks&s\\-C_Ks&t\end{smallmatrix}\right]
\left[\begin{smallmatrix}0\\1/N\end{smallmatrix}\right]}(\tau_K)^n
\quad\textrm{by the isomorphism in (\ref{reciprocity}) and (F3)}\\
&=&h_{\left[\begin{smallmatrix}s/N\\t/N\end{smallmatrix}\right]}(\tau_K)^n,
\end{eqnarray*}
from which we get
\begin{equation}\label{hzetah}
h_{\left[\begin{smallmatrix}0\\1/N\end{smallmatrix}\right]}(\tau_K)=
\zeta h_{\left[\begin{smallmatrix}s/N\\t/N\end{smallmatrix}\right]}(\tau_K)
\quad\textrm{for some $|n|$th root of unity}.
\end{equation}
Since $h_{\left[\begin{smallmatrix}0\\1/N\end{smallmatrix}\right]}(\tau_K)$
and $h_{\left[\begin{smallmatrix}s/N\\t/N\end{smallmatrix}\right]}(\tau_K)$ belong to $K_{(N)}$ by
(\ref{CM}), we deduce by Lemma \ref{rootofunity} that $\zeta$ is a $12N$th root of unity.
Thus we obtain by (\ref{hzetah}) that
\begin{equation*}
h_{\left[\begin{smallmatrix}0\\1/N\end{smallmatrix}\right]}(\tau_K)^{12N}=
h_{\left[\begin{smallmatrix}s/N\\t/N\end{smallmatrix}\right]}(\tau_K)^{12N},
\end{equation*}
which implies $d_N(\tau_K)=0$. But, this contradicts that $K$ does not belong to $S_N$.
\par
Therefore, we conclude that if $h_{\left[\begin{smallmatrix}0\\1/N\end{smallmatrix}\right]}(\tau_K)$ is nonzero, then $h_{\left[\begin{smallmatrix}0\\1/N\end{smallmatrix}\right]}(\tau_K)^n$
generates $K_{(N)}$ over $H_K$ for any nonzero integer $n$.
\end{proof}

\begin{remark}\label{invariantremark}
Let $K$ be an imaginary quadratic field of discriminant $d_K$, and let $n$ be a nonzero integer.
\begin{itemize}
\item[(i)] Every weakly holomorphic function in $\mathcal{F}_1$ is a polynomial in $j(\tau)$ over $\mathbb{Q}$ (\cite[Theorem 2 in Chapter 5]{Lang}).
Moreover, since $\mathrm{ord}_q~j(\tau)=-1$ (\cite[p. 45]{Lang}), we
see that $d_N(\tau)$ is a polynomial in $j(\tau)$ over $\mathbb{Q}$ of degree
$|\mathrm{ord}_q~d_N(\tau)|$. It is well known that $j(\tau_K)$ generates $H_K$ over $K$
(as we mentioned), and $[H_K:K]\rightarrow\infty$ as $|d_K|\rightarrow\infty$
(\cite[p. 149]{Cox}). Hence, if $|d_K|$ ($\geq7$) is large enough so as to have
$[H_K:K]>|\mathrm{ord}_q~d_N(\tau)|$, then $K$ does not belong to the set $S_N$.
\item[(ii)] Let $g(\tau)=g_{\left[\begin{smallmatrix}1/N\\0\end{smallmatrix}\right]}(\tau)^{12Nn}
g_{\left[\begin{smallmatrix}0\\1/N\end{smallmatrix}\right]}(\tau)^{24Nn}$ be the function
stated in Theorem \ref{usingSiegel}. By making use of the Kronecker second limit formula one can also
show that if $\gcd(N,\,3\cdot5\cdot 7\cdot 13\cdot d_K(d_K-1))=1$, then
$g(\tau_K)$ generates $K_{(N)}$ over the ground field $K$
instead of $H_K$ (see \cite{K-Y}).
\end{itemize}
\end{remark}

\bibliographystyle{amsplain}

\address{
Applied Algebra and Optimization Research Center\\
Sungkyunkwan University\\
Suwon-si, Gyeonggi-do 16419\\
Republic of Korea} {hoyunjung@skku.edu}
\address{
Department of Mathematical Sciences \\
KAIST \\
Daejeon 34141\\
Republic of Korea} {jkkoo@math.kaist.ac.kr}
\address{
Department of Mathematics\\
Hankuk University of Foreign Studies\\
Yongin-si, Gyeonggi-do 17035\\
Republic of Korea} {dhshin@hufs.ac.kr}

\end{document}